\documentclass[11pt]{amsart}
\usepackage{amssymb,latexsym}
\usepackage{enumerate}
\usepackage{graphicx}
\input xy
\xyoption{all}

\makeatletter
\@namedef{subjclassname@2010}{%
  \textup{2010} Mathematics Subject Classification}
\makeatother

\newtheorem{thm}{Theorem}[section]
\newtheorem{cor}[thm]{Corollary}
\newtheorem{lem}[thm]{Lemma}
\newtheorem{prob}[thm]{Open problem}
\newtheorem{prop}[thm]{Proposition}

\theoremstyle{definition}
\newtheorem{defin}[thm]{Definition}
\newtheorem{rem}[thm]{Remark}
\newtheorem{exa}[thm]{Example}

\numberwithin{equation}{section}

\newcommand{\zz}[1]{}

\newcommand{\R}{\mathbb{R}}

\newcommand{\Z}{\mathbb{Z}}
\DeclareMathOperator{\Sph}{S}

\begin{document}


\baselineskip=17pt

\title[A new approach to the equivariant topological complexity]{A new approach to the equivariant topological complexity}

\author[W. Lubawski]{Wojciech Lubawski}
\address{Theoretical Computer Science Department\\ Faculty of Mathematics and Computer
Science\\ Jagiellonian University\\ 30-348 Krak\'{o}w, Poland\\
$\&$  Institute of Mathematics, Polish Academy of Science\\
ul. \'Sniadeckich 8, 00-956 Warszawa, Poland}
\email{w.lubawski@gmail.com}

\author[W. Marzantowicz]{Wac\l aw Marzantowicz$^{*}$}

\address{Faculty of Mathematics and Computer Science, Adam Mickiewicz
University of Pozna\'n, ul. Umultowska 87, 61-614 Pozna\'n, Poland}
\email{marzan@amu.edu.pl}

\thanks{$^{*}$Supported  by the Polish Research Grant NCN 2011/03/B/ST1/04533
}

\date{}

\begin{abstract}
We present a  new approach to an equivariant version of Farber's
topological complexity called invariant topological complexity. It
seems that the presented  approach is more adequate for the analysis
of impact of a symmetry on a motion planning algorithm than the one
introduced and studied by Colman and Grant. We show many bounds for
the invariant topological complexity comparing it with already known
invariants and prove that in the case of a free action it is equal
to the topological complexity of the orbit space. We define
the Whitehead version of it. {\vskip -0.5cm}
\end{abstract}



\maketitle

\; \; {\vskip -0.5cm}
\section{Introduction}
A topological invariant introduced by Farber in \cite{farber1,
farber2} and called topological complexity was the first to
estimate complexity of a motion planning algorithm. With a
configuration space $X$ of a mechanical system he associated a
natural number $TC(X)$ -- the topological complexity of $X$. To
be more precise he considered the following natural fibration
\begin{equation}\label{eq1}\pi\colon PX \rightarrow X\times X\end{equation}
from the free path space on $X$ which assigns to a path $\gamma$
defined on the unit interval its ends $(\gamma(0), \gamma(1))$. The
topological complexity is the least $n$ such that $X\times X$ can be
covered by $n$ open sets $U_1 , \ldots , U_n$ such that for each $i$
there is a homotopical section $s_i\colon U_i\rightarrow PX$ to
$\pi$.
 This invariant is a special version of a well known
Lusternik-Schnirelmann (or LS for short) category of $X\times X$ (cf.
\cite{cornea} for more detailed exposition of this notion and
other references and Lemma \ref{2:5} for the description of topological complexity in the language of LS category).

In this paper we discuss  the following question: If the mechanical
system admits a symmetry with respect to a compact Lie group (hence similar applies for the configuration space $X$) what is an
appropriate definition  of the topological complexity that takes
into account that symmetry? An answer is not that simple as it may
look like and is not unique. We define an invariant, different than the equivariant topological complexity introduced by Colman and Grant in \cite{colman-grant}, called the invariant topological complexity. By showing its properties we would like to
demonstrate that in many situations it better suits into the given frame than that of
\cite{colman-grant}.

Let $G$ be a compact Lie group and $X$ be a $G$ space (therefore we
assume that $G$ is the \emph{symmetry group} that appears in $X$). The
definition of  topological complexity uses the natural fibration
\ref{eq1}. If the space $X$ admits a $G$-action then $PX$ is a
$G$-space with the pointwise action and so does $X\times X$ via the diagonal action. It would
be natural to define the equivariant complexity by assuming that all
maps considered are $G$-maps. This approach has been studied in
\cite{colman-grant}. We will use the notation introduced there and
 denote the invariant by $TC_G(X)$. In spite of its mathematical
naturalness this approach has some disadvantages that we present
below.
\begin{figure}[htb]
\renewcommand*{\figurename}{Picture}
\centering
\includegraphics[height=5cm]{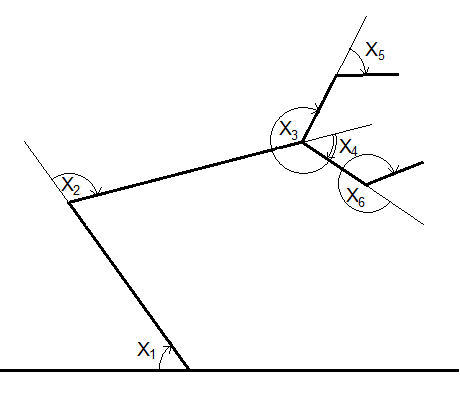}
\caption{A mechanical robot arm}
\label{fig:1}
\end{figure}

Let us consider a mechanical robot arm as shown in Picture \ref{fig:1}. We associate to it a configuration space $X:=\{ (X_1,X_2,X_3,X_4,X_5,X_6)\,|\; X_i\in \Sph^1 \}$. The mechanical robot arm admits a $G := \mathbb Z / 2=\{ 1,t \}$ symmetry as showed in Picture \ref{fig:2}. In other words the element $t$ acts as a map 
$$t\colon X\ni(X_1,X_2,X_3,X_4,X_5,X_6)\mapsto (X_1,X_2,X_4,X_3,X_6,X_5)\in X.$$ 
\begin{figure}[htb]
\renewcommand*{\figurename}{Picture}
\centering
\includegraphics[height=5cm]{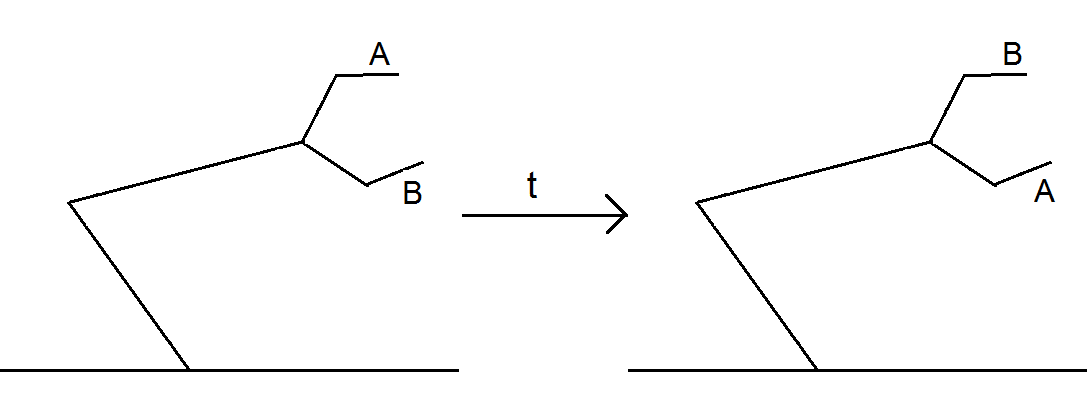}
\caption{A symmetric robot arm with an action of $t$}
\label{fig:2}
\end{figure}
exchanging the part $A$ of the arm with the part $B$.

Assume now that we are given a path $\xi$ between points $x$ and $y$ in the configuration space $X$ as presented in Picture \ref{fig:3}.
\begin{figure}[htb]
\renewcommand*{\figurename}{Picture}
\centering
\includegraphics[height=5cm]{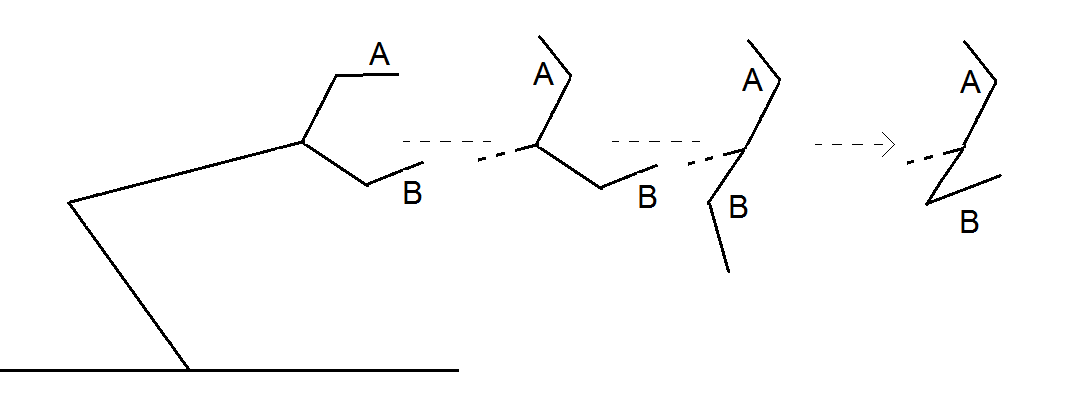}
\caption{A path $\xi$ in the configuration space between states $x$ and $y$}
\label{fig:3}
\end{figure}

Note that although points $x$ and $tx\in X$ are distinct in the configuration space there is no physical difference between these two states of a mechanical robot arm as can be observed from Picture \ref{fig:2}. Therefore if we are supposed to \emph{equivariantly} choose a path between $tx$ and $ty$ it should be equal to $t\xi$. This is the natural requirement that leads us to the definition of $TC_G(X)$.

Here we claim a stronger statement. Assume we are given a path $\xi$ between states $x$ and $y$. For such a choice in order to be \emph{equivariant} we require it determines not only a path between $tx$ and $ty$, that is $t\xi$, but also "a path" $\eta$ between $x$ and $ty$ as well as "a path" $t\eta$ between $tx$ and $y$ (we put the quotient marks to stress the fact that at this point we do not know if such choices can be realized by continuous paths). In other words we should exploit the $G\times G$-structure of the space $X\times X$. Our main problem is that usually $PX$ is not a $G\times G$-space and we will show in Section \ref{2.} how to overcome this difficulty and define invariant topological complexity denoted by $TC^G(X)$.

Surprisingly our approach seems to be better suited to serve as an equivariant version of topological complexity since it has more natural mathematical properties. For example if the group $G$ acts freely on $X$ then in general $TC_G(X)\neq TC(X/G)$ where $X / G$ is
the orbit space and $TC(X/G)$ is the topological complexity of
$X/G$. We will show that in our case $TC^G(X) = TC(X/G)$. 

A  bridge to apply advanced homotopy theory in the theory of
Lusternik-Schnirelmann category is the Whitehead version of it (cf.
\cite{whitehead1}). We will show that for the invariant topological complexity we can define a Whitehead version of it and for a finite group $G$ it gives the original invariant topological complexity. We conjecture  that the same  holds for any compact Lie group. Finally we provide examples which distinguish the equivariant topological complexity and the invariant topological complexity and calculate the latter in several cases.

Throughout this paper we assume that $X$ is a compact $G$-ANR (cf.
\cite{murayama} for the properties of $G$-ANRs). The class of
$G$-ANRs includes $G$-ENRs (cf. \cite{jaworowski} for the
definition), countable $G$-CW complexes, smooth $G$-manifolds with a smooth
action of $G$, etc.
\section{Lusternik-Schnirelmann category}

\subsection{Basic definitions}
In this section we define and give basic properties of a
version of an equivariant Lusternik-Schnirelmann category for
topological spaces that we will use later on in our considerations.
We shall use the standard notations of compact group transformation
theory as presented in \cite{bredon}.

%

\begin{defin}\label{1:1}
Let $A\subseteq X$ be a closed $G$-subset of a $G$-space $X$. An open
$G$-subset $U\subseteq X$ will be called \emph{$G$-compressible} into
$A$ whenever the inclusion map $\iota_U\colon U\subseteq X$ is
$G$-homotopic to a $G$-map $c\colon U\rightarrow X$ such that
$c(U)\subseteq A$.
\end{defin}
\noindent This allows us to define our main tool.
\begin{defin}[cf. \cite{clpu} for the nonequivariant case]\label{1:2}\label{Adfi3}
For a given $G$-subset  $A\subseteq X$ the
$A$-Lusternik-Schnirelmann $G$-category of a $G$ space $X$ is the
least $n$ such that $X$ can be covered by $U_1,\ldots , U_n$ open
$G$-subsets of $X$ each $G$-compressible into $A$. We denote in by
$\hbox{}_Acat_G(X)$.
\end{defin}

Note that this version is closely related to the standard Lusternik-Schnirelmann category as defined for example in \cite{cornea}. Recall that a $G$-space $X$ is $G$-connected if for every closed subgroup $H\subseteq G$ the space $X^H$ is path connected. 

\begin{rem}\label{1:3}
If $X$ is path connected and the action of $G$ on $X$ is trivial then
$_{\{\ast\}} cat_G(X) = cat(X)$ for every $\ast\in X$. If
$X$ is $G$-connected and $\ast\in X^G$ then $_{\{\ast\}} cat_G(X) =
cat_G(X)\,$ where $cat_G(X)$ denotes the equivariant
Lusternik-Schnirelmann category of $X$ (cf. \cite{marzan}).
\end{rem}

This version of the LS category has many properties similar to the standard LS category. We say  that a pair of $G$-spaces $(X,A)$ $G$-dominates $(Y,B)$ if  there are $G$-maps
$$f\colon (X,A)\rightarrow (Y,B)\text{ and }g\colon (Y,B)\rightarrow (X,A)$$ such that $fg\simeq id_{(Y,B)}$ in the equivariant topological category of
 pairs of $G$-spaces.

\begin{prop}\label{1:4}
If $(X,A)$ $G$-dominates $(Y,B)$ then $\hbox{}_Acat_G(X) \geqslant \hbox{}_Bcat_G(Y)$.
\end{prop}

\begin{proof}
The proof is similar to that of Lemma 1.29 in \cite{cornea} after a suitable change of categories.
\end{proof}

\subsection{The Whitehead definition of the category}

Analogously  to the classical LS category the  notion of
$G$-category defined in \ref{Adfi3}
  has its Whitehead version.

Recall that a pair of $G$-spaces $A\subseteq X$ is a closed $G$-cofibration (or the Borsuk pair) whenever $A$ is $G$-invariant and closed in $X$ and it has the equivariant homotopy extension property, i.e. for any $G$-space $Y$,
$G$-homotopy $h: A \times I \to Y$ and a $G$-map $f\colon X\rightarrow Y$ such that $h|A\times\{ 0 \} = f|A$ there is an equivariant
homotopy $H: X \times I \to Y$ extending $h\cup f$.

\begin{defin}\label{1:7}\label{2:13}\label{G-Borsuk pair} Let $A\subseteq X$ be a closed $G$-cofibration.
By a \it{fat $A$-sum} we mean for every $n\in \mathbb N$\linebreak a $G$-space
$F_A^n(X) \subseteq X^{n}:=X\times \ldots \times X$ defined as
follows:
\begin{itemize}
\item $F_A^1 := A$

\item $F_A^n(X)$ is the colimit in the category of $G$-spaces of the following diagram:
$$\xymatrix{A\times F^{n-1}_A(X)\ar[rr]\ar[d] && X\times F^{n-1}_A(X)\ar[d]\\
A\times X^{n-1}\ar[rr] && F_A^n(X)
}$$

\end{itemize}
\end{defin}

\begin{defin}\label{1:8}
We say that the $G$-Whitehead $A$-category, denoted by $\hbox{}_Acat^{Wh}_G
(X)$, is less or equal $n$ if and only if there is a $G$-map $\xi_n
\colon X\rightarrow F^n_A(X)$ such that the following diagram is homotopy
commutative:
$$
\xymatrix{
X\ar[drr]^{\Delta_n} \ar[rr]^{\xi_n} && F^n_A(X)\ar[d]^{\subseteq}\\
&& X^n}
$$
where $\Delta_n\colon X\rightarrow X^n$ is the diagonal map.
\end{defin}

\begin{thm}\label{thm1}\label{1:9}
Let $X$ be a $G$-space and $A\subseteq X$ closed
$G$-cofibration. Then \newline \centerline{$_Acat_G(X) =
\hbox{}_Acat^{Wh}_G(X)$.}
\end{thm}

\noindent Before giving the proof of the theorem (which mimics the proof of Theorem 1.55 in \cite{cornea}) we need a technical lemma:

\begin{lem}\label{1:10}
Under the assumptions of the theorem, if $\{U_i\}_{i=1}^n$ is an invariant open covering of $X$ such that for each $i$ there exists $G$-map $s_i\colon U_i\rightarrow A$ such that $G_i\colon \iota\circ s_i$
is $G$-homotopic to $(U_i\subseteq X)$ then there exist an open and invariant covering $\{V_i\}_{i=1}^n \leqslant \{U_i\}_{i=1}^n$ such that for each $i$ there exists a $G$-homotopy $H_i\colon X\times I\rightarrow X$
with $H_i(x,0) = x$ for each $x\in X$ and $H_i(x,1) = \iota\circ s_i$ for $x\in V_i$.
\end{lem}
\begin{proof}
Here $\{V_i\}_{i=1}^n \leqslant \{U_i\}_{i=1}^n$ means that for
every $1\leq i\leq n$,  $V_i \subset U_i$. By a direct argument,  we
can find invariant coverings $\{V_i\}$ and $\{W_i\}$ of $X$ such
that
$$V_i\subseteq \bar{V_i}\subseteq W_i\subseteq \bar{W_i}\subseteq
U_i.$$  Since $X$ is a $G$-ANR, $X/G$ is normal. Moreover
$(\bar{V}_i/G) \cap ((X\setminus W_i)/G)  =\emptyset$ in $X/G$.
Consequently, by normality of $X$ there exists  a $G$-invariant
continuous function $\lambda\colon X\rightarrow I$ be such that
$\lambda(\bar{V_i}) = 1$ and $\lambda (X\backslash W_i) = 0$. For
each $i$ we define the $G$-homotopy by:
$$H_i\colon X\times I\ni (x,t)\mapsto \begin{cases} x & ,x\in X\backslash W_i \\ G_i(x , t\cdot \lambda (x)) & ,x\in \bar{W_i}\end{cases}$$
\end{proof}

\begin{rem}\label{1:11}
Of course the converse implication in the above lemma also holds.
Given a family of $G$-homotopies $H_i$ with an invariant covering
$\{V_i\}$ of $X$ it is sufficient to set $s_i := H_i(- ,t)|_{V_i}$.
\end{rem}

\begin{proof}[Proof of Theorem \ref{thm1}]
If $_Acat_G(X)\leqslant n$ then we have $n$ $G$-homotopies
$H_i\colon X\times I\rightarrow X$ satisfying conditions of the
lemma \ref{1:10}. Now to show that $_Acat^{Wh}_G(X)\leqslant n$ it is
sufficient to put $$\xi_n\colon X\ni x\mapsto (H_i(x,1))_{i=1}^n\in
F^n_A.$$

Conversely, if $_Acat^{Wh}_G(X)\leqslant n$ then we are given
$\xi_n\colon X\rightarrow F^n_A(X)$ such that $\Delta$\ and
$(F^n_A(X)\subseteq X^n)\circ \xi_n$ are $G$-homotopic by a homotopy
$\zeta$. We denote by $\zeta^i$ the $i$-th coordinate of $\zeta$.
Since $A\subseteq X$ is a $G$-cofibration then there exists $N=N(A)$
an invariant and open neighborhood of $A$ in $X$ such that $A$ is a
$G$-deformation retract of $N$. Let us denote this equivariant
deformation retraction by $R$. Then $R(x,0) = x$ and $R(x,1)\in A$
for $x\in N$. Set $U_i := H ^{-1}(N,1)$. It is easy to see that
$\{U_i\}$ is an invariant open covering of $X$. Moreover setting
$$H_i\colon X\times I\ni \mapsto \begin{cases} \zeta^i(x, 2t) & ,0\leqslant t\leqslant 1/2\\ R(\zeta^i(x,1),2t-1) & ,1/2\leqslant t \leqslant 1.\end{cases}$$
we obtain the required family of $G$-homotopies with $s_i\colon U_i\rightarrow A$ equal to $H_i(-,1)|_{U_i}$.
\end{proof}

 A natural analog of the definition of $n$-connectedness  in the equivariant case is the
following:

\begin{defin}[comp. \cite{may}, Definition I.2.1]\label{G-n-connected} We call a $G$-space $X$  $G$-$n$-connected if $X^G\neq\emptyset$ and $\pi_i(X^H) = 0$ for
all $i \leqslant  n$ and all closed subgroups $H\subseteq G$. Likewise, a
$G$-pair $(X,A)$ is $n$-connected if $A^G\neq\emptyset$ and $\pi_i(X^H,A^H) = 0$ for all $i
\leqslant n$ and all closed subgroups $H\subseteq G$.
\end{defin}

The following fact is well-known in the non-equivariant case, e.g.
\cite[Proposition 4.13 and Corollary 4.16]{hatcher}.
Since we could not find any direct reference of the equivariant case we reprove the CW approximation theorem as stated in \cite[theorem XI.3.6]{may} making some minor changes.

\begin{prop}\label{Corollary 4.16} If $(X,A)$ is an $G$-$n$-connected $G$-CW pair for a discrete $G$ then
there exists a $G$-CW pair $(Z;A) \sim (X;A) \; rel\;  A$ such that all
cells of $Z \setminus A$ have dimension greater than $n$.
\end{prop}

\begin{proof}
We construct a family of $G$-CW complexes $A\subseteq Y_0\subseteq Y_1\subseteq Y_2\subseteq \ldots $ together with maps $\gamma_i\colon (Y_i,A)\rightarrow (X,A)$ such that $\pi_q(\gamma_i)$ is a surjection for $q>i+n$
and an isomorphism for $q\leqslant i+n$. Let $b\in A^G$. We choose a representative map $$s^q_H\colon\mathbb (I^{q+1},\partial I^{q+1})\rightarrow (X^H,A^H)$$ for each element of $\pi_q (X^H, A^H, b)$ where $q>0$ and $H$
runs over conjugacy classes of closed subgroups of $G$. Let $\tilde{Y_0}$ be the disjoint union of spaces $G/H\times I^{q}$, one for each chosen $s^q_H$ and of $A$. Let $\tilde{\gamma_0}$ be a map induced by all $s^d_H$.
For each $s^q_H\colon\mathbb (I^{q},\partial I^{q})\rightarrow (X^H,A^H)$ we identify each $x\in\partial I^{q}$ with $s^q_H(x)\in A$ in $\tilde{Y_0}$ hence obtaining
$$\gamma_0\colon Y_0:= (\bigsqcup_{s^q_H} G/H\times I^{q}) \cup_{\sqcup {s^q_H | \partial I^{q}}} A.$$
Note that $\gamma_0^H$ is an isomorphisms on $\pi_i$ for $i\leqslant n$ and a surjection for $i>n$. Indeed, surjectivity is obvious for all $i$, for injectivity for $i\leqslant n$ let $H$ be a closed subgroup of $G$. Note that
$(Y_0)^H$ is of the form $(\bigsqcup_{s^q_{H'}} I^q) \cup_{\sqcup s^q_{H'}| \partial I^q} A^H$ where the sum is taken over all $s^q_{H'}$ such that there is a $G$-map $G/H\rightarrow G/H'$. We have a cofibration
$$A^H\rightarrow (Y_0)^H\rightarrow (\bigvee_{s^q_{H'}} \ S^q)$$
which shows that $H_i((Y_0)^H, A^H)$ for $i\leqslant n$. Therefore by \cite[corollary 7.10]{whitehead} we get that $\pi_i((Y_0)^H , A^H) = 0$.

Now assume that $\gamma_i\colon (Y_i,A)\rightarrow (X,A)$ has been constructed. We choose representative maps $(f,g)$ for each pair of elements in $\pi_{q}((Y_i)^H,A^H, b)$ that are equalized by $\pi_q(\gamma_i)$
(note that then $q>i+n$). Using the cellular approximation theorem \cite[theorem 3.4]{may} we assume that $f$ and $g$ have image in the $q$ skeleton of $Y_i$. Let $(Y_{i+1},A)$ be the homotopy coequalizer
of the disjoint union of all such maps -- i.e. $(Y_{i+1},A)$ is obtained by attaching $G/H_+\wedge (I^q,\partial I^q)\times I_+$ via each chosen pair. Note that such operation does not affect $\pi_\ast(Y_i,A)$
for $\ast\leqslant d+i$ and kills the kernel of $\pi_{i+d+1} (\gamma_i)$. We define $\gamma_{i+1}$ with the use of homotopies $\gamma_if\simeq \gamma_ig$ based at $b$. Now it is enough to triangulate $Y_{i+1}$ as
a $G$-CW complex containing $Y_i$.
We set $(Z,A) = \cup_i (Y_i,A)$ and $\gamma=\cup\gamma_i\colon (Z,A)\rightarrow (Y,A)$
\end{proof}

\begin{rem}\label{pawalowski}
Note that without the assumption that the fix point set is nonempty the above theorem does not hold.
For a finite group $G$ R. Oliver \cite{oliver} showed the existence
of a $G$-space $X$ satisfying\linebreak the following  properties (we are grateful to K. Pawa{\l}owski for recalling us this example): {$X$ a $G$-CW
complex},  {for every solvable subgroup $H\subseteq G$ the set $X^H$
is contractible}, {for every non-solvable subgroup $K\subseteq G$
the set $X^K$ is empty}. In particular one can take $G =
A_5\subseteq \Sigma_5$, the group of pair symmetries on five element
set. It is well known that $A_5$ is non-solvable but every its
proper
 subgroup is solvable.

Note that for every subgroup $H\subsetneq G$ the $G$-space $X^H$ is now
contractible. If the assertion of Proposition \ref{Corollary 4.16} held
this would lead  to the existence of $0$-dimensional $G$-CW complex
$Y$, $G$-homotopy equivalent to $X$ and contractible since
$\pi_n(Y)=\pi_n(Y^e)= \pi_n(X^e)=0$ for all $n\geq 0$. But this
implies that $Y=\ast$ is a point and therefore $Y=Y^G$ which gives us
a contradiction as $X^G=\emptyset$.
\end{rem}

For a compact $G$-CW complex $X$ by $\dim_G X $  we mean the maximal
dimension  $n$ of  $G$-cells of the form  $G/H \times [0,1]^n$ that
appear in the construction of $X$. Consequently $\dim_G X = \dim
X/G$. If $G$ is discrete then of course $\dim_G X = \dim X = \dim
X/G$.

\begin{thm}\label{1:12}
Let $G$ be a finite group and  $A\subseteq X$ be a pair of
$G$-CW-complexes. If the pair $(X,A)$ is $G$-$n$-connected then
$_Acat(X)\leqslant \dim_G X / (n+1) +1$.
\end{thm}

\begin{proof}
By Proposition \ref{Corollary 4.16} we may assume that $X\backslash
A$ has no $k$ dimensional $G$-cells for $k\leqslant n$. Then
$F^s_A(X)$ and $X$ have similar $s(n+1)-1$ skeleton. Let $s$ satisfy
$(s-1)(n+1) \leqslant \dim_G X \leqslant s(n+1)$ and using the
equivariant cellular approximation (comp. \cite[Theorem 3.4]{may})
theorem we get that the diagonal map $\Delta_s\colon X\rightarrow
X^s$ is $G$ homotopic to a $G$ cellular map $\xi\colon X\rightarrow
F^s_A(X)$.
\end{proof}

At the end of this subsection we prove a technical lemma that will
be used later on to prove the product formula for the category
-- Theorem \ref{1:5}).

\begin{lem}\label{1:15}
Let $X$ and $Y$ be $G$-spaces and $A\subseteq X$ and $B\subseteq Y$ its closed $G$-subsets. Then there is a commutative diagram
$$\xymatrix{
F^n_A(X)\times F^m_B(Y)\ar[rr]\ar[d] && F^{n+m-1}_{A\times B}(X\times Y)\ar[d]\\
X^n\times Y^m \ar[rr]^{\omega^{n,m}} && (X\times Y)^{n+m-1}}$$
such that $\omega^{n,m}\circ(\Delta_n(X)\times \Delta_m(Y)) = \Delta_{n+m-1}(X\times Y)$.
\end{lem}

\begin{proof}
We prove the theorem inductively. If $(n,m) = (n,1)$ then our diagram is of the form
$$\xymatrix{
F^n_A(X)\times B\ar[rr]^\alpha\ar[d] && F^{n}_{A\times B}(X\times Y)\ar[d]\\
X^n\times Y \ar[rr]^\omega^{n,1} && (X\times Y)^{n}}$$ and it is
easy to see that it is commutative whenever we set $\;\; \alpha (x_1
, \ldots , x_n , b)= (x_1 , \ldots , x_n , b,\ldots , b)$.  The
condition on the diagonal is also satisfied. Similar argument prove
the statement for $(n,m) = (1,m)$. Now let us assume that
$n,m\geqslant 2$. Since in the category of $G$ CW complexes the
product of two pushouts is a pushout of products therefore we have a
pushout
$$\xymatrix{
A\times B\times F^n_A(X)\times  F^m_B(Y)\ar[rr]\ar[d] && X\times Y\times F^n_A(X)\times  F^m_B(Y)\ar[d]\\
A\times B\times X^n \times Y^m \ar[rr] && F^{n+1}_A(X) \times  F^{m+1}_B(Y)}.$$
We get a commutative diagram
\begin{center}
\resizebox{1\textwidth}{!}{$$\xymatrix{
A\times B\times F^n_A(X)\times  F^m_B(Y)\ar[rr]\ar[dd]\ar[rd] && A\times B\times F^{n+m-1}_{A\times B}(X\times Y)\ar[dd]\ar[rd] & \\
& X\times Y\times F^n_A(X)\times  F^m_B(Y)\ar[rr]\ar[dd] &&X\times Y\times F^{n+m-1}_{A\times B}(X\times Y)\ar[dd]\ar[dd]\\
A\times B\times X^n \times Y^m \ar[rr]\ar[rd] && A\times B \times (X\times Y)^{n+m-1}\ar[rd] &\\
& F_A^{n+1}(A)\times F_B^{m+1}(B)\ar@{-->}[rr]^{\eta} && F^{n+m}_{A\times B}(X\times Y)}.$$}
\end{center}
where $\eta$ is the universal map between two pushouts. The whole diagram is over the map
$$ \omega^{n+1,m+1}:= id_{X\times Y}\times \omega^{n,m}\colon X\times Y \times X^n\times Y^m\rightarrow X\times Y\times (X\times Y)^{n+m-1}$$
and the assertion follows.
\end{proof}

\subsection{Bounds for the category}

\begin{lem}\label{1:13}
Let $X$ be a $G$-set, $H\subseteq G$ closed subgroup and assume that $A\subseteq B$ are its closed invariant subsets. Then:
\begin{enumerate}[1)]
\item $_Bcat_G(X)\leqslant \hbox{}_Acat_G(X)$;
\item $_Acat_G(X)\leqslant \hbox{}_Bcat_G(X)\cdot \hbox{}_Acat_G(B)$;
\item $_{A/G}cat(X/G)\leqslant \hbox{}_Acat_G(X)$;
\item $_Acat_H(X)\leqslant \hbox{}_Acat_G(X)$;
\end{enumerate}
\end{lem}
\begin{proof}
For the proof of point 2) let us assume that $_Bcat_G(X)\leqslant n$ and
$_Acat_G(B)\leqslant m$. Let $U_1,\ldots ,U_n$ be open invariant
subsets of $X$, each compressible by $s_i$ into $B$ and $V_1,\ldots
,V_m$ open invariant subsets of $B$, each compressible by $t_j$ into
$A$. Let $$W^j_i := s_i^{-1}(V_j).$$ We know that $\{W^j_i\}$ is an
invariant open covering of $X$. We define $r^j_i := t_j\circ s_i
|_{W^j_i}$ then it can be readily seen that $W^j_i$ is compressible
into $A$ by $r^j_i$. Since the cardinality of $\{W^j_i\}$ is $n\cdot
m$ the inequality follows. The rest of the points are obvious.
\end{proof}

The category behaves well (i.e. similar to the standard LS category) under taking products.

\begin{thm}\label{1:5}
Let $X$ and $Y$ be two $G$-spaces, $A\subseteq X$, $B\subseteq Y$ their closed $G$-subsets. Then
$$_{A\times B}cat_G(X\times Y)\leqslant \hbox{}_Acat_G(X) + \hbox{}_Bcat_G(Y) - 1$$
where on $X\times Y$ is given the diagonal action of $G$.
\end{thm}

\begin{proof}
We prove the theorem using lemma \ref{1:15}. Note that whenever we have homotopy commutative diagrams
$$
\xymatrix{
X\ar[drr]^{\Delta_n(X)} \ar[rr]^{\xi_n} && F^n_A(X)\ar[d]^{\subseteq} & Y\ar[drr]^{\Delta_m(Y)} \ar[rr]^{\xi'_m} && F^m_B(Y)\ar[d]^{\subseteq}\\
&& X^n &&& Y^m}
$$
Then homotopy commutative is also the following diagram
$$
\xymatrix{
X\times Y \ar[drr]\ar[rr]^{\xi_n\times \xi'_m} && F^n_A(X)\times F^m_B(Y)\ar[d]^{\subseteq}\ar[rr]  && F^{n+m-1}_{A\times B} (X\times Y)\ar[d]\\
&& X^n\times Y^m \ar[rr]^{\omega^{n,m}} && (X\times Y)^{n+m-1}}
$$
Now $\omega^{n,m}\circ(\Delta_n(X)\times \Delta_m(Y)) = \Delta_{n+m-1}(X\times Y)$ which ends the proof.
\end{proof}

\begin{rem}
Note that we do not need any additional assumption for the action of
$G$ on $X$ and $Y$. In case $A$ and $B$ are singletons
and $X$ and $Y$ are $G$-connected our result is equivalent to that
obtained in \cite[Theorem 3.15]{colman-grant} nevertheless our
approach is much more general.
\end{rem}

\begin{thm}\label{1:6}
Let $X$ be a $G$-space and $Y$ be a $H$-space for Lie groups $G$ and $H$. Then for a closed $G$-subset $A\subseteq X$ and a closed $H$-subset $B\subseteq Y$ we have (where we consider $X\times Y$ as a standard $G\times H$-space)
$$_{A\times B}cat_{G\times H}(X\times Y)\leqslant \hbox{}_Acat_G(X) + \hbox{}_Bcat_H(Y) - 1.$$
\end{thm}

\begin{proof}
Follows directly from \ref{1:5} since we can consider $X$ as a $G\times H$ space with trivial $H$ action and $Y$ as a $G\times H$ space with trivial $G$ action.
\end{proof}

We end this section with an observation which relates $H$ and
$G$-categories.

\begin{prop}\label{1:16}
Let $X$ be a $G$-space, $A\subseteq X$ its closed $G$-subset, $H$ closed subgroups of $G$ then
$$\hbox{}_{A^H}cat_H(X^H) \leqslant \hbox{}_{A}cat_G(X).$$
\end{prop}

\begin{proof}
If $U\subseteq X$ is $G$-compressible into $A$ then $V:= U\cap X^H$ is $H$-compressible into $A^H$ (which follows from the equivariant condition for the homotopy).
\end{proof}

\section{Topological robotics in presence of a symmetry}\label{2.}

\subsection{Equivariant topological complexity}

Let $X$ be a topological space with an action of a compact Lie group $G$. Consider the space of all continuous paths $s\colon I\rightarrow X$ with the compact open topology and denote it by $PX$. It admits a natural action of $G$.

Observe that the natural projection $$ \pi\colon PX\ni s\mapsto
(s(0),s(1))\in X\times X $$ is a continuous, $G$-equivariant
$G$-fibration. Whenever we talk about $X\times X$ we consider it as
a $G$-space (via the diagonal action) unless explicitly stated.

Recall that, by a motion planning algorithm on an open set
$U\subseteq X\times X$ we mean a local section $s\colon U \rightarrow PX$
of the fibration $\pi$, i.e. $\pi\circ s = (U\subseteq X\times X)$
(cf. \cite{farber1,farber2}).

\begin{defin}\label{dfi1}\label{2:2}
An equivariant motion planning algorithm on an open set $U\subseteq X\times X$ is a $G$-equivariant local section $s\colon U \rightarrow PX$
of the $G$-fibration $\pi$, i.e. $\pi\circ s =(U\subseteq X\times X)$.
\end{defin}

Recall that the topological complexity of $X$, denoted by
$TC(X)$, is the smallest $n$ such that $X\times X$ can be covered by
$U_1,\ldots ,U_n$ -- open subsets such that for each $i$ there
exists $s_i\colon U_i\rightarrow PX$ a motion planing algorithm on
$U_i$ (cf. \cite{farber1, farber2}).

\begin{defin}[\cite{colman-grant}]\label{dfi2}\label{2:3}
The equivariant topological complexity, denoted by $TC_G(X)$, of a
$G$-space $X$ is the smallest $n$ such that $X\times X$ can be
covered by $U_1,\ldots ,U_n$ -- invariant open subsets such that for
each $i$ there exists  $s_i\colon U_i\rightarrow PX$ an equivariant
motion planing algorithm on $U_i$.
\end{defin}

\begin{exa}\label{2:4}
Let $X = G := \Sph^1$ with $G$ acting by multiplication from the
left. Note that $X/G$ is trivial so that $TC(X/G) = 1$ whereas
$p\colon (\Sph^1)^I\rightarrow \Sph^1\times \Sph^1$
cannot have a section, in particular cannot have an equivariant one,
so $TC_G(X)\geqslant 2$. This shows that the topological complexity of
an orbit space and the equivariant topological complexity does not have
to coincide, even in the simplest examples of a free action.
\end{exa}
Our aim is to give a suitable definition of a motion planning
algorithm in an equivariant setting which induces an invariant motion
planning algorithm and as mentioned in the introduction have a
reasonable geometric meaning. Moreover we want this motion planning
algorithm to define a topological complexity which coincides with the
topological complexity of an orbit space for free $G$-spaces. In
order to do so we will translate it into the language of the
Lusternik-Schnirelmann category.

Recall that $\Delta_n\colon X\rightarrow X^n$ is the diagonal map.
Let us denote by $\Delta(X)$ the image of $\Delta_2$ in $X\times X$.

\begin{rem}
Let $X$ be a $G$ space. The map $\pi\colon PX\rightarrow X\times X$
is a $G$-fibration (satisfies the homotopy lifting property for $G$-maps).
\end{rem}


\begin{lem}[comp. \cite{colman-grant}]\label{2:5}
\label{lmm1} For a $G$-space $X$ the following statements are equivalent:
\begin{enumerate}[1)]

\item $TC_G(X) \leqslant n$;
\item there exist $n$ invariant open sets $U_1,\ldots ,U_n$ which cover $X\times X$ and $\bar{s_i}\colon U_i\rightarrow PX$ such that $p\circ \bar{s_i}$ is $G$-homotopic to $U_i\rightarrow X\times X$;
\item $_{\Delta (X)}cat_G (X\times X)\leqslant n$.
\end{enumerate}
\end{lem}

\begin{proof}
\noindent 1)$\Rightarrow$2) is obvious.

\noindent 2)$\Rightarrow$1). Let $s\colon U\rightarrow PX$ be such
that $H\colon p\circ s\simeq (\colon U\subseteq X\times X)$ (as
$G$-maps), where $p\colon PX\rightarrow  X\times X$. Then from the
equivariant homotopy
lifting property there exists a $G$-homotopy
$\hat{H}\colon U\times I\rightarrow PX$ such that the following
diagram is commutative:
$$\xymatrix{
U\times \{ 0\}\ar[rr]^s\ar[d]^\subseteq && PX\ar[d]^p\\
U\times I \ar[urr]^{\hat{H}} \ar[rr]^H && X\times X
}$$
now it is sufficient to set $\bar{s}(a,b):= \hat{H}(a,b;1)$.

\noindent 2)$\Leftrightarrow $3). Let $H\colon PX\times I\rightarrow PX$ be given as:
$$H(\omega; t) (s) = \omega (s(1-t))\text{ for }\omega\in PX,\, s,t\in I.$$ It is a $G$-deformation retraction  between $PX$ and $\iota (X)\subseteq PX$, where $\iota(x)\equiv x$ assigns to
 every point $x\in X$ the constant map defined by it; $\iota$ in this case is a $G$-homeomorphism onto the image.
Composing $\bar{s}\colon U\rightarrow PX$  with $p\circ H_1$ we get
$\hat{s}\colon U\rightarrow \Delta(X)$ which is $G$-homotopic in $X$
to the inclusion $U\subseteq X\times X$. On the other hand, given
$\hat{s}\colon U\rightarrow \Delta(X)$ we have $\hat{s} = \Delta_2\circ
\hat{s}'$, where $\hat{s}'\colon U\rightarrow X$ is a $G$-map.
Composing  it with $\iota$ we get $\bar{s}\colon U\rightarrow PX$
such that $p\circ \bar{s}$ is homotopic in $X\times X$ to the
inclusion $U\subseteq X\times X$. Note that $\Delta_2 = p\circ
\iota$. These processes are mutually inverse up
 to $G$-homotopy so that we proved the equivalence.
\end{proof}

\zz{Hence we obtained a characterization of topological complexity
in terms of a suitable version of LS-category.}

\subsection{Invariant topological complexity}

The main problem arising from geometric interpretation is that $PX$ is not a $G\times G$-space -- which is equivalent to the problem that $\Delta(X)$ is not a $G\times G$-subspace of $X\times X$. Fortunately the latter can be easily fixed.

For a given $G$-space  $X$ we denote by $\daleth (X)$ the saturation of $\Delta(X)$ with respect to the $G\times G$-action:
$$\daleth (X):= (G\times G)\cdot\Delta(X)\subseteq X\times X.$$
Instead of $\Delta(X)$ in the definition of equivariant topological complexity we consider $\daleth (X) \subset X\times X$ and instead of considering open subsets $G$-compressible into
 $\Delta(X)$ we consider open subsets of $X\times X$ which are $G\times G$-compressible into $\daleth(X)$.

\begin{defin}\label{dfi4}\label{2:6}
For a $G$-space $X$ we define invariant topological complexity as
$$TC^G(X) = _{\daleth(X)}cat _{G\times G} (X\times X).$$
\end{defin}

Let us state a lemma similar to \ref{lmm1} but formulated for the invariant topological complexity. For a $G$ space $X$ we consider
$$PX\times_{\daleth(X)} PX := \{ (\gamma,\delta)\in PX\times PX\colon G\cdot\gamma(1) = G\cdot \delta(0) \}$$ as a $G\times G$ space with the obvious multiplication $(g_1,g_2)\cdot (\gamma ,\delta) = (g_1\gamma ,g_2\delta)$.
Note that we have a natural $G\times G$ map $p\colon PX\times_{\daleth(X)} PX \rightarrow X\times X$ given by $p(\gamma ,\delta) = (\gamma(0), \delta (1))$.

\begin{prop}\label{fibration}
Let $X$ be a $G$ space. The map $p\colon PX\times_{\daleth(X)} PX\rightarrow X\times X$
is a $G\times G$-fibration (satisfies the homotopy lifting property for $G\times G$-maps).
\end{prop}

\noindent{\it Proof.} Note that $\{0,1\}\subseteq I$ is a closed $G\times G$-cofibration where we consider $\{0,1\}$ and $I$ as trivial $G\times G$ spaces.
 Therefore from \cite[Theorem 7.8]{whitehead}  the following map
$$\overline{p}\colon (X\times X)^{I}\rightarrow X\times X\times X\times X$$
is a $G\times G$-fibration, where $\overline{p}(f) =
(f(0),f(1))$. We also know that a projection to any of the factors in the product of two $G\times G$-spaces is a $G\times G$-fibration
and that a composition of two $G\times G$-fibrations is a $G\times
G$-fibration. Now it is enough to note that
$$\;\;\;\;\;\;\;\;\;\;\;\;\;\;\;\;\;\;\;\;\;\;\;\;\;\;\;\;\;\;\;\; p = pr_{1,4}\circ \overline{p} | _{\overline{p}^{-1}(X\times \daleth(X)\times X)}\colon \overline{p}^{-1}(X\times \daleth(X)\times X)\rightarrow X\times
X\,.\;\;\;\;\;\;\;\;
\;\;\;\;\;\;\;\;\;\;\;\;\;\;\;\;\;\;\;\;\;\;\;\;\Box$$

\begin{lem}\label{2:7}
\label{lmm2} For a $G$-space $X$ the following statements are equivalent:
\begin{enumerate}[1)]
\item $TC^G(X) \leqslant n$;
\item there exist $n$ $G\times G$ invariant open sets $U_1,\ldots ,U_n$ which cover $X\times X$ and $G\times G$
 maps $\bar{s_i}\colon U_i\rightarrow PX\times_{\daleth(X)} PX$ such that $p\circ \bar{s_i}$ is $G\times G$-homotopic to $id$ (as maps $U_i\rightarrow X\times X$);
\item $_{\daleth (X)}cat_{G\times G} (X\times X)\leqslant n$.
\end{enumerate}
\end{lem}

\begin{proof}
\noindent 1) $\Leftrightarrow$ 3) by the definition. \noindent 3) $\Rightarrow$ 2). Let $\iota_U\colon U\subseteq X\times X$ be an open $G\times G$ invariant subset that in $G\times G$ compressible into $\daleth(X)$. Let $H\colon \iota_U\simeq c$ be a $G\times G$
homotopy where $c(U)\subseteq \daleth(X)$. From the equivariant homotopy lifting property we get that
$$\xymatrix{
U\times \{ 0\}\ar[rr]^s\ar[d]^\subseteq && PX\times_{\daleth(X)} PX\ar[d]^p\\
U\times I \ar[urr]^{\hat{H}} \ar[rr]^H && X\times X
}$$
where $s(u_1,u_2) = (c_{u_1},c_{u_2})$ and $c_u$ is the constant path equal to $u$. Now it is enough to set $s_i := \hat{H}(-,1)$.

\noindent 2) $\Rightarrow $ 3). Let $H\colon PX\times_{\daleth(X)} PX\times I\rightarrow PX\times_{\daleth(X)} PX$ be given as:
$$H(\gamma, \delta,t)(\tau,\tau ') = (\gamma (\tau + t(1-\tau)), \delta((1-t)\tau ) ).$$
It is a $G\times G$ deformation retraction between $PX\times_{\daleth(X)} PX$ and $\iota(\daleth(X))\subseteq PX\times_{\daleth(X)} PX$ where $\iota$ assigns to $(u_1,u_2)$ the constant maps defined by it.
 Then if $s_i\colon U\rightarrow X$ is a $G\times G$ map such that $F\colon p\circ s_i\simeq id_ U$ for a $G\times G$ homotopy $F$ then $U$ is $G\times G$ compressible into $\daleth (X)$ as
$$id_ U\simeq p\circ s_i \sim p\circ id_{PX\times_{\daleth(X)PX}} \circ s_1\simeq p\circ H(-,1)\circ s_i$$
and $H(-,1)\circ s_i\colon U\rightarrow \daleth(X)$.
\end{proof}

\begin{rem}\label{2:8}
For a $G$-space $X$ we have inequality $TC(X/G)\leqslant TC^G(X)$.
\end{rem}
One of our main requirements was that our version of equivariant topological complexity of $X$ should be equal to the topological complexity of the orbit space $X/G$. The invariant topological complexity satisfies this condition.

\begin{thm}\label{2:9}
\label{thm0}
Let $X$ be a free $G$-space.  Then $TC(X/G) = TC^G(X)$.
\end{thm}

\zz{Let us recall the Covering Homotopy Theorem of Palais (cf.
\cite{bredon}):  {\it Let $G$ be a compact Lie group, $X$, $Y$
$G$-spaces, $f\colon X\rightarrow Y$ a $G$-equivariant map. Denote
by $f'\colon X/G\rightarrow Y/G$ the map induced by $f$. Let
$F'\colon X/G\times I\rightarrow Y/G$ be a homotopy which preserves
the orbit structure and starts at $f'$. Then there exists an
equivariant homotopy $X\times I\rightarrow Y$ covering $F'$ starting
at $f$.}\label{2:10}}

\begin{proof}
Assume that $TC(X/G)\leqslant n$ then there exists a $G\times G$ invariant covering $U_1,\ldots ,U_n$ of $X\times X$ and $s_i\colon U_i\rightarrow \Delta(X/G)$ such that $s_i$ is homotopic to
 $U_i\subseteq X$ via the homotopy $H\colon U_i\times I\rightarrow X\times X$ (we assume it starts at the identity). Since the action of $G\times G$ is free on $X\times X$,
 the homotopy $H$ preserves the orbit structure. Hence from the  Covering Homotopy Theorem of Palais (cf. \cite{bredon}) we get a $G\times G$-equivariant homotopy $\tilde{H}\colon U_i \times I\rightarrow X\times X$
 starting at $U_i\subseteq X\times X$. For the orbit map $\pi\colon X\times X\rightarrow X/G\times X/G$ we have $\pi^{-1}(\Delta(X/G)) = \daleth(X)$ hence $G\times G$-map
 $\tilde{s_i}\colon U_i\rightarrow \daleth(X)$ can be defined by the formula $\tilde{s_i}(z) = \tilde{H}(z,1)$.
\end{proof}

As a direct consequence of  computation of the topological
complexity of real projective space by Farber, Tabachnikov and
Yuzvinsky \cite{farber-tab-yuzv} we get the following
\begin{cor}\label{symmetric complexity for projective space}

If $n\neq 1,3,7$ then $TC^{\Z / 2}(\Sph^n)$, where $\Z /2$ acts on $\Sph^n$ by the origin symmetry, is equal to the smallest
$k$ for which $\R P^n$ admits an immersion in $\R^{k-1}$.
\end{cor}

\subsection{Whitehead invariant topological complexity}

From the classical theory (cf. \cite{dugundji} for the
non-equivariant case, \cite{lewis} for short explanation how to pass
to the equivariant one) we get that for a $G$-CW complex $X$ the map
$$\Delta(X)\subseteq X\times X$$
 is a closed $G$-cofibration. Nevertheless the
case of $$\daleth(X)\subseteq X\times X$$ and a question if it is a
$G\times G$-cofibration is much more complicated. We do not know the
answer for a general compact Lie group $G$. Here we will prove it
for a finite $G$.

\begin{cor}\label{2:11}
From the theorem \ref{thm1}, lemma \ref{lmm1} and the remark above we get that
$$TC_G(X) = _{\Delta(X)}cat_G^{Wh}(X\times X)$$
In particular for the classical topological complexity we get that
$$TC(X) = _{\Delta(X)}cat(X\times X) = _{\Delta(X)}cat^{Wh}(X\times X).$$
Moreover if $\daleth(X)\subseteq X\times X$ is a closed $G\times G$-cofibration then
$$TC^G(X) = _{\daleth(X)}cat_G(X\times X) = _{\daleth(X)}cat_{G\times G}^{Wh}(X\times X).$$
\end{cor}

Now let us proceed to the proof that $\daleth(X)\subseteq X\times X$ is a $G\times G$ cofibration for a finite $G$ for which we will use an equivalent formulation.

\begin{prop}\label{2:14}
\label{equivalent definitions}
Let $G$ be a compact Lie group. A compact metrizable $G$-space $X$ is a $G$-ANR if and only if every pair $(Y,X)$ is a closed $G$-cofibration for $Y$ and $X$ metrizable, $X$ closed and invariant in $Y$.
\end{prop}

\begin{prop}[Jaworowski  \cite{jaworowski}]\label{2:15}\label{Jaworowski}
Let $G$ be a compact Lie group.  A compact $G$-space is a $G$-ANR if
and only if for every closed subgroup $H\subseteq G$ the fixed point
set $X^H$ is an ANR.
\end{prop}

To shorten the notation let us denote by $\tilde{X}$ the product $X\times X$. Recall that
$\tilde{X}$ posses a natural action of the group
$\tilde{G}:=G\times G$ induced by the action of $G$ on $X$.

\begin{thm}\label{2:16} Let $G$ be a finite group and $X$ a compact $G$-ANR. Then $\daleth(X)$ is a $\tilde{G}$-ANR.
\end{thm}

\begin{proof}
First observe that $\daleth(X)$ can be represented as the saturation
of $\Delta(X)\subseteq \tilde{X}$ with respect to the action of group
$G_1:=G\times\{e\} \subseteq \tilde{G}$, i.e.
$$\daleth(X)= \{(gx,x)\colon \;\ g\in G, \;x\in X\} $$
Indeed, since $G_1\subseteq G$, $\{(gx,x)\colon \;\ g\in G, \;x\in X\}  \subset \daleth(X)$.
On the other hand any $z=(x_1,x_2)=(g_1x,g_2x)$ can be represented
as $(\tilde{g}x,x)\}$ where $ y=g_2x$ and $\tilde{g}=g_1\,g_2^{-1}$.
This shows that $\daleth(X)\subseteq \{(gx,x)\}$.

Of course $\daleth(X)$ is $\tilde{G}$-invariant closed subset of
$\tilde{X}$ as the  image of compact space $\tilde{G}\times
\Delta(X)$. In view of the Jaworowski theorem (\ref{Jaworowski}) it
is enough  to show that for every $\tilde{H}\subseteq \tilde{G}$ the
space $\daleth(X)^{\tilde{H}}$ is an ANR.

Let $h=(h_1,h_2) \in \tilde{H}$. A point $(gx,x)$ belongs to
$\daleth(X)^h$ (or equivalently to $X^{\{h\}}$, $\{h\}$ the cyclic
group generated by $h$) if and only if $h(gx,x) = (gx,x)$. The latter is equivalent to $h_2 x=x$ and $h_1gx=gx$. The first equality gives $x\in X^{h_2}$, and the second $ gx \in X^{h_1}$.
Since $G_{gx}= g G_x g^{-1}$ the latter means that $x\in X^{g^{-1}h_1g^{-1}}$.
Consequently $(gx,x)\in \daleth(X)^h$ if and only if $x\in X^{h_2} \cap X^{g^{-1}h_1
g}$.

Next note that $X^{h} \cap X^{h'} = X^{\{h,h'\}}$, where $\{h,h'\}$
is a subgroup generated by $h$ and $h'$.

\zz{Indeed since $h\subseteq \{h,h'\}$, $h'\subseteq \{h,h'\}$,
$X^{\{h,h'\}} \subseteq X^{h}$ and $X^{\{h,h'\}} \subseteq X^{h'}$,
thus $ X^{\{h,h'\}} \subseteq X^{h} \cap X^{h'}$.

Conversely, if $ x\in  X^{h} \cap X^{h'}$ then $x\in X^{h_1^{i_1^1}
h_2^{i_2^1} \,\cdots \, h_1^{i^1_r}h_2^{i^2_r}}$ for any word
$h_1^{i_1^1} h_2^{i_2^1} \,\cdots \, h_1^{i^1_r}h_2^{i^2_r}$. This
means that $x\in X^{\{h,h'\}}$, thus $X^{h} \cap X^{h'} \subset
X^{\{h,h'\}}$.}

From it follows that  given $g\in G$  for  $\daleth(X)_g:=
\{(gx,x)\}, \; x\in X$  and $h=(h_1',h_2)\in \tilde{H}$ the fixed
point set $\daleth(X)_g^h$ is equal to $ \daleth(X)_g^h
=X^{\{h',h_2\}} $,  where $h'=gh_1g^{-1}$. But such a set is an ANR,
because $X$ is a $G$-ANR.

Observe next that
\begin{equation}\label{intersection}
\daleth(X)_{g_1}^h \cap \daleth(X)_{g_2}^h = X^{h_2}\cap X^{h_2}
\cap X^{g_1^{-1}h_1g_1} \cap X^{g_2^{-1}h_1g_2} =X^{\{h',h'',
h_2\}}\,,
\end{equation}
with $ h'=g_1^{-1}h_1g_1$ and $h''=g_2^{-1}h_1g_2$. Consequently, it is an ANR.

Since $\daleth(X)= {\underset{g\in G}\cup} \daleth(X)_g$, $G$ is
finite, we know that $\daleth(X)^h$ is a finite union of ANRs
such that the intersections of each two of them are  ANRs.  From the
well known property of ANRs: \emph{"$X$, $Y$, $X\cap Y$ are ANRs
implies that $X\cup Y$ is an ANR "} it follows that $\daleth(X)^h$
is an ANR.

Now lets take $h=(h_1,h_2)$, and $h'=(h_1',h_2')$. For a given $g\in G$
\begin{equation}\label{second intersection}
\daleth(X)_{g}^h \cap \daleth(X)_{g}^{h'} =  X^{h_2}\cap X^{h_2'}
\cap X^{g h_1g^{-1}} \cap X^{gh_1'g^{-1}} =X^{\{h_2,h_2', g
h_1g^{-1}\,,gh_1'g^{-1}\}}
\end{equation}
Observe that for any $\tilde{G}$-subset $A$  of $\tilde{X}$ we have
$A^H = {\underset{h\in H}\cap} A^h $.

Now let $h^1=(h^1_1,h^1_2), \,\dots,\, h^s=(h^s_1,h^s_2)$ be all
elements of $H\subset \tilde{G}$. For a given $g\in G$
$$(\daleth(X)_g)^H= {\underset{h\in H}\cap}(\daleth(X)_g)^h=
X^{\{h^1_2,h^2_2, ..., h^s_2,\, gh^1_1g^{-1}, ...,\,
gh^s_1g^{-1}\}}$$
and consequently this set is an ANR, since $X$ is a $G$-ANR.

Finally, by the same argument $\daleth(X)^H = {\underset{g\in
G}{\cup}} (\daleth(X)_g)^H$ is an ANR, because for two $g_1,\,g_2
\in G$ we have
$$(\daleth(X)_{g_1})^H \cap (\daleth(X)_{g_2})^H=X^{\{h^1_2,h^2_2, ..., h^s_2, \,g_1h^1_1g_1^{-1}, ...\,,\, g_1h^s_1g_1^{-1}, g_2h^1_1g_2^{-1}, ...\,,
\,g_2h^s_1g_2^{-1}\}} $$ This shows that $(\daleth(X)_{g_1})^H \cap
(\daleth(X)_{g_2})^H$ is an ANR and  so is the union
${\underset{g\in G}{\cup}} (\daleth(X)_g)^H = \daleth(X)^H$.
\end{proof}

\begin{rem}\label{2:17} Let $H\subseteq \tilde{G}= G\times G$, $p_1:G\times G\to
G$, $p_2:G\times G \to G$ the projections onto the corresponding
coordinates. Set $H_1:=p_1(H)$, $H_2:=p_2(H)$ and let $ \hat{H}:=
H_1\times H_2\subset \tilde{G}$.

Since the subgroup $\hat{H}$ contains $H$ we have an
inclusion
\begin{equation}\label{inclusion}
 A^{\hat{H}} \subseteq A^H
\end{equation}
for any $\tilde{G}$ subset $A$ of
$\tilde{X}$. Since the inclusion $H\subseteq \hat{H}$ is strict in general, the inclusion \ref{inclusion} is in general strict as well.
\end{rem}

\begin{prob}\label{2:18}
Is it true that $$\daleth (X)\subseteq X\times X$$ is always a closed $G\times G$-cofibration for a compact $G$-CW complex $X$ and a compact Lie group $G$?
\end{prob}

The above problem seems be  difficult in general.

\subsection{Bounds for the invariant topological complexity}

We start with a product formula for the invariant and equivariant topological complexity.

\begin{thm}
Let $X$ be a $G$-space and $Y$ a $H$-space. We consider $X\times Y$ as a $G\times H$-space. Then
$$TC_{G\times H}(X\times Y)  \leqslant TC_G(X)+TC_H(Y)-1$$
Moreover,  if $\daleth(X)\subseteq X\times X$ is a $G\times
G$-cofibration and $\daleth(Y)\subseteq Y\times Y$ is a $H\times
H$-cofibration then
$$TC^{G\times H}(X\times Y)  \leqslant TC^G(X)+TC^H(Y)-1$$
\end{thm}

\begin{proof}
Since $\Delta(X\times Y) = \Delta(X)\times \Delta(Y)$ and
$\daleth(X\times Y)= \daleth(X)\times \daleth(Y)$ this is a direct
consequence of Corollary \ref{1:6}.
\end{proof}

\begin{cor}\label{prod}
Let $X$ and $Y$ be any $G$ spaces. Then for $X\times Y$ considered
as a $G$-space with the diagonal action we have that
$$TC_G(X\times Y)  \leqslant TC_G(X)+TC_G(Y)-1\,.$$
\end{cor}

\begin{proof}
Since $\Delta(X\times Y) = \Delta(X)\times \Delta(Y)$ it is a
 consequence of theorem \ref{1:5}.
\end{proof}

\begin{rem}
The above corollary is not true for $TC^G$ -- let $X:=\Sph^1$ be a free $G:=\Sph^1$-space. Since the action is free we obtain from theorem \ref{thm0} that $TC^G(X) = TC(X/G)=  1$. On the other hand the space $X\times X$ is a free $G$-space via the diagonal action so that $TC^G(X\times X) = TC(X\times X / G)\cong TC(\Sph^1) = 2$.
\end{rem}

\begin{rem}
Note that the corollary \ref{prod} generalizes Theorem 4.2 of \cite{gonzales} where a product formula for $TC_G$ is proved under more restrictive hypotheses.
\end{rem}


From our point of view one of the most important properties of the invariant topological complexity is that it is indeed finite for a large family of $G$ spaces $X$.
We have an obvious inequality
$$TC(X)\leqslant cat(X\times X).$$
We will show that it passes to the equivariant case. For completeness let us first recall

\begin{prop}[\cite{colman-grant}, Proposition 5.6]\label{3:2}
If $X$ is $G$-connected then $TC_G(X) \leqslant cat_{G}(X\times
X)$.
\end{prop}

We give a similar result concerning the invariant topological complexity.

\begin{prop}\label{3:3}
If $X$ is $G$-connected then $TC^G(X) \leqslant _{O\times
O}cat_{G\times G}(X\times X)$ where $O = G\cdot x_0$ for some
element $x_0$ in $X$.
\end{prop}

\begin{proof}
Let $U$ be a set $G\times G$-compressible into $O\times O$. We have
a $G\times G$-homotopy $F\colon U\times I\rightarrow X\times X$ such
that $F\colon id_U\simeq c$ where $c(U)\subseteq O\times O$. Let $H
= G_{x_0}\times G_{x_0}$. We know that $p((PX\times
_{\daleth(X)}PX)^H ) = (X\times X)^H$ which follows from the $G$-connectivity of $X$ hence $p(\gamma,\delta) = (x_0,x_0)$. Then we
define $s\colon U\rightarrow PX\times_{\daleth (X)}PX$ by
$s(y_0,y_1) = (g_0,g_1)\cdot (\gamma,\delta)$ whenever $c(y_0,y_1) =
(g_0,g_1)\cdot (x_0,x_0)$. Now note that $p\circ s \simeq c \simeq
id_U$.
\end{proof}

\begin{rem}
The above theorem allows us to show that $TC^G(X)$ is in many cases
finite. In particular, if $x_0\in X^G$ and $X$ is $G$-connected then
 $_{\{x_0\}\times \{x_0\}}cat _{G\times G} (X\times X) \leqslant
2\hbox{}_{\{x_0\}}cat_{G} (X) - 1 = 2cat_G (X) - 1$ by Theorem
\ref{1:6}.
\end{rem}

Equivariant and invariant topological complexity share some basic homotopical properties.

\begin{prop}\label{3:4}
Let $X$ $G$-dominates $Y$, that is there are $f\colon X\rightarrow
Y$ and $f^\prime\colon Y\rightarrow X$ such that $f\, f^\prime\simeq
id_Y$ are $G$-homotopic. Then
$$TC_G(X)\geqslant TC_G(Y)\text{ and } TC^G(X) \geqslant TC^G(Y).$$
\end{prop}

\begin{proof} The part concerning $TC_G(X)$ can be found in \cite[Theorem 5.2]{colman-grant}. 

For the proof for $TC^G(X)$ let $H\colon f\circ f^\prime\simeq id_Y$ be
the $G$-homotopy. Note that then
$$H\times H\colon (X\times X,\daleth(X))\times I \rightarrow (Y\times Y,\daleth(Y))$$ is a $G\times G$-homotopy between
$(f\times f)\circ (f^\prime\times f^\prime)$ and $id_{(Y\times Y,
\daleth (Y))}$. Now the assertion follows from \ref{1:4}.
\end{proof}

\begin{cor}\label{3:5}
For a $G$-set $X$ we have
$$TC(X^G)\leqslant TC_G (X)\text{ and } TC(X^G)\leqslant TC^G(X).$$
\end{cor} 

\begin{proof} The part concerning $TC_G(X)$ follows from
\cite[Proposition 5.3]{colman-grant}.

For the proof for $TC^G(X)$ first note that $TC^G(X^G) = TC(X^G)$ and 
$$\daleth(X)^{G\times G} = \daleth (X^G) = (G\times G)
\Delta(X^G)= \Delta(X^G)$$ therefore from Proposition \ref{1:16} we obtain
$$TC^G(X^G) = _{(G\times G) \Delta(X^G)}cat_{G\times G}(X^G\times
X^G) \leqslant   \hbox{}_{\daleth (X)}cat_{G\times G}(X\times X) =
TC^G(X)$$
which ends the proof.
\end{proof}

\section{Examples of calculations}

We end this article with calculations of $TC^G(X)$ in some basic examples.

\begin{exa}
Let $G$ act on itself by left translations. The action of $G$ is free and therefore from theorem \ref{2:9} we get that
$$TC^G(G) = TC(G / G) = TC(\ast) = 1$$
which is in contrast to the case of equivariant topological
complexity where we have that $TC_G(G) = cat(G)$ (comp.
\cite{colman-grant}, Theorem 5.11).
\end{exa}

\begin{exa}
Let $\mathbb Z / 2 = \{1 , \tau \}$ act on $\Sph^n$, $n\geqslant 1$ by reflecting the last coordinate. Note that for $n=1$ the set $(\Sph^1)^{\mathbb Z /2}$
is disconnected so that $TC_{\mathbb Z /2}(\Sph^1) = TC^{\mathbb Z /2} (\Sph^1) = \infty$. If $n>1$ then $\Sph^n$ is $\mathbb Z /2$ connected so
that $$TC^{\mathbb Z /2}(\Sph^n)\leqslant cat_{\mathbb Z /2 \times \mathbb Z /2} (\Sph^n\times \Sph^n)\leqslant 2cat_{\mathbb Z /2}(\Sph^n) - 1 = 3$$
 by theorem \ref{3:3}. On the other hand since $(\Sph^n)^{\mathbb Z /2}\cong \Sph^{n-1}$ Corollary \ref{3:5} implies that $TC^{\mathbb Z /2}(\Sph^n)\geqslant TC(\Sph^{n-1}) = 3$ for $n$ odd.

For an even $n$ let $U_1\subseteq \Sph^n\times \Sph^n$ be defined
as follows $U_1 = \{(x,y)\in(\Sph^n)^2\colon x\neq -y\text{ if
}x,y\in\Sph^{n-1}\}$. Then for each $(x_1,x_2)\in U_1$ there is
a unique shortest path $s'(x_1,x_2)$ joining $x_1$ or $\tau x_1$ and
$x_2$ or $\tau x_2$ in the upper hemisphere. Let $s_1(x_1,x_2)=
(\alpha_1 s|_{[0,\tfrac{1}{2}]}, \alpha_2s|_{[\tfrac{1}{2} , 1]})$
in case we were joining $\alpha_1x_1$ with $\alpha_2x_2$ for
$\alpha_i\in\mathbb Z / 2$.

Let $U_2\subseteq \Sph^n\times \Sph^n$ be defined as follows. We set $V_2
:= \{(x,y)\in(\Sph^n)^2\colon x, y\in\Sph^{n-1},\; x\neq
y\}$.  Now $V_2$ has a small $\mathbb Z/2\times \mathbb Z/2$ invariant open
neighborhood $U_2$ in $\Sph^n\times \Sph^n$ such that the projection
$\pi\colon U_2\rightarrow V_2$ into the equator $\Sph^{n-1}$ is
$\mathbb Z /2\times \mathbb Z /2$
 equivariant deformation retraction. We define for each $(x_1,x_2)$ a path from $x_1$ to $x_2$ as follows. First choose a non vanishing vector field $\nu$ on $\Sph^{n-1}$. The path $s'_2(x_1,x_2)$ consists of four parts. First by the shortest path we move $x_1$ to $\pi(x_1)$, then using the shortest path we move $\pi(x_1)$ to $-\pi(x_2)$ and using the vector field $\nu$ we move $-\pi(x_2)$ to $\pi(x_2)$ using a spherical arch defined by $\nu (\pi(x_2))$ and we end by moving through the shortest path $\pi(x_2)$ to $x_2$. We obtain $s$ from $s'$ by cutting it into two parts.

As it can be easily checked these two sets satisfy the definition of the invariant topological complexity and prove that $TC^{\mathbb Z / 2} (\Sph^n) = 2$ for $n$ even.

Note that we have $TC_{\mathbb Z / 2} (\Sph^n) = 3$ for $n>1$
as shown in \cite{colman-grant},Example 5.9.
\end{exa}

\end{document}